\documentclass{article}
\usepackage{amsmath}
\usepackage{amssymb}
\usepackage{amsthm}
\usepackage{cite}
\usepackage{enumerate}
\usepackage{ragged2e}
\usepackage{tikz-cd}
\usepackage{verbatim}
\usepackage{xypic}
\title{Day convolution for $\infty$-categories}
\author{Saul Glasman}

\newcommand{\angs}[1]{\langle #1 \rangle}
\newcommand{\bb}{\mathbb}
\newcommand{\colim}[1]{\underset{#1}{\text{colim }}}
\newcommand{\del}{\partial}
\newcommand{\D}{\Delta}
\newcommand{\eps}{\epsilon}
\newcommand{\F}{\mathcal{F}}
\newcommand{\Fun}{\text{Fun}}
\newcommand{\Ga}{\Gamma}
\newcommand{\id}{\text{id}}
\newcommand{\inc}{\subseteq}
\newcommand{\inj}{\hookrightarrow}
\newcommand{\iy}{\infty}
\newcommand{\la}{\lambda}
\newcommand{\La}{\Lambda}
\newcommand{\mb}{\mathbf}
\newcommand{\mc}{\mathcal}
\newcommand{\Map}{\text{Map}}
\newcommand{\ol}{\overline}
\newcommand{\op}{\text{op}}
\newcommand{\os}{\overset}
\newcommand{\ot}{\otimes}
\newcommand{\td}{\widetilde}
\newcommand{\toe}{\overset{\sim} \to}
\newcommand{\X}{\times}

\newcommand{\Cat}{\mb{Cat}}
\newcommand{\Top}{\mathbf{Top}}

\theoremstyle{definition}

\newtheorem{dfn}[subsection]{Definition}
\newtheorem{lem}[subsection]{Lemma}
\newtheorem{ntn}[subsection]{Notation}
\newtheorem{obs}[subsection]{Observation}
\newtheorem{prop}[subsection]{Proposition}
\newtheorem{rec}[subsection]{Recollection}
\newtheorem{thm}[subsection]{Theorem}

\begin{document}
\maketitle

\begin{abstract}
Given symmetric monoidal $\infty$-categories $\mb{C}$ and $\mb{D}$, subject to mild hypotheses on $\mb{D}$, we define an $\infty$-categorical analog of the Day convolution symmetric monoidal structure on the functor category $\Fun(\mb{C}, \mb{D})$. An $E_\infty$ monoid for the Day convolution product is a lax monoidal functor from $\mb{C}$ to $\mb{D}$. \end{abstract}

\section{Introduction}

Let $(\mc{C}, \ot_\mc{C})$ and $(\mc{D}, \ot_\mc{D})$ be two symmetric monoidal categories such that $\mc{D}$ admits all colimits. In \cite{Day}, Day equips the functor category $\Fun(\mc{C}, \mc{D})$ with a ``convolution" symmetric monoidal structure: If $F, G : \mc{C} \to \mc{D}$ are functors, then their convolution product $F \ot_\text{Day} G$ is defined as the left Kan extension of $\ot_\mc{D} \circ (F \X G) : \mc{C} \X \mc{C} \to \mc{D}$ along $\ot_\mc{C} : \mc{C} \X \mc{C} \to \mc{C}$. According to \cite[Example 3.2.2]{Day}, the commutative monoids for the convolution product are exactly the lax monoidal functors from $\mc{C}$ to $\mc{D}$.

An important special case of the Day convolution is the tensor product of Mackey functors; see, for example, \cite{PS}. In a recent paper \cite{Bar14}, Barwick develops a theory of higher-categorical Mackey functors in order to study equivariant $K$-theory. To work with such objects, and more generally to study the multiplicative structure of $K$-theory (see \cite{Bar13}) it will be useful to develop a higher-categorical analog of the Day convolution product. This is the purpose of this note, which we accomplish in Section \ref{abawe}. In particular, we show in Proposition \ref{laxmon} that $E_\infty$ algebras for the Day convolution product are lax symmetric monoidal functors - that is, $\infty$-operad maps in the sense of \cite[Definition 2.1.2.7]{HA} - thus answering a question of Blumberg. In section \ref{Fonata}, we construct the Yoneda embedding for a symmetric monoidal category as a symmetric monoidal functor into the presheaf category equipped with the Day convolution product, and deduce that our construction agrees in this special case with the object constructed by Lurie in \cite[Corollary 6.3.1.12]{HA}. 

We'd like to acknoweledge the influence of numerous helpful conversations with Clark Barwick, Denis Nardin, and particularly Jay Shah, who read this paper carefully and discovered substantial errors in an earlier version.

\section{The Day convolution symmetric monoidal $\iy$-category}\label{abawe}

\begin{ntn} Throughout this paper, $\F$ will denote the category of finite pointed sets and pointed maps, or by abuse of notation, the nerve of that category. For convenience, we recall some definitions related to $\F$, all of which are used frequently in \cite{HA}. We write $\angs{n}$ for the object $\{*, 1, \cdots, n\}$ of $\F$. If $S \in \F$ is an object, then $S^o$ denotes the finite set $S \setminus \{\ast\}$. A morphism $f : S \to T$ in $\F$ is called \emph{inert} if it induces a bijection between $f^{-1}(T^o)$ and $T^o$; $f$ is called \emph{active} if it's surjective and $f^{-1}(\ast) = \ast$.
\end{ntn}
 Let $\mb{C}^\ot \to \F$ and $\mb{D}^\ot \to \F$ be symmetric monoidal $\infty$-categories (see \cite[Definition 2.0.0.7]{HA}). To sidestep potential set-theoretic issues, we'll fix a strongly inaccessible uncountable cardinal $\lambda$ and assume that both $\mb{C}^\ot$ and $\mb{D}^\ot$ are $\lambda$-small. If $k : K \to \F$ is a map of simplicial sets, then we denote by $\mb{C}^\ot_k$ the pullback

\[\xymatrix{\mb{C}^\ot_k \ar[r] \ar[d] & \mb{C}^\ot \ar[d] \\
K \ar[r]^k & \F.
}\]
In particular, if $f$ is any morphism in $\F$, $\mb{C}^\ot_f$ is the pullback

\[\xymatrix{\mb{C}^\ot_f \ar[r] \ar[d] & \mb{C}^\ot \ar[d] \\
\D^1 \ar[r]^f & \F.
}\]
and if $S$ is an object of $\F$, $\mb{C}^\ot_S$ is the fiber of $\mb{C}^\ot$ over $S$.

Suppose $k :K \to \F$ is an arbitrary map of simplicial sets. We define a simplicial set 
\[\ol{\Fun(\mb{C}, \mb{D})^\ot}\]
 by the following universal property: there is a bijection, natural in $k$,
\[\Fun_{\F}(K, \ol{\Fun(\mb{C}, \mb{D})^\ot}) \toe \Fun_\F (\mb{C}^\ot_k, \mb{D}^\ot).\]
\begin{obs}
A vertex of $ \ol{\Fun(\mb{C}, \mb{D})^\ot}$ is a finite set $S$ together with a functor $\mb{C}^\ot_S \to  \mb{D}^\ot_S$, which is to say a functor 
\[F_S : \mb{C}^S \to \mb{D}^S.\]
 Similarly, an edge of $ \ol{\Fun(\mb{C}, \mb{D})^\ot})$ is given by a morphism $f : S \to T$ in $\F$ together with a functor 
\[ F_f :  \mb{C}^\ot_f \to \mb{D}^\ot_f\]
over $\D^1$. A section of the structure morphism $\ol{\Fun(\mb{C}, \mb{D})^\ot} \to \F$ corresponds to a map over $\F$ from $\mb{C}^\ot$ to $\mb{D}^\ot$.
\end{obs}
Suppose $\mb{D}$ has all colimits. We seek to prove that $\ol{\Fun(\mb{C}, \mb{D})^\ot} \to \F$ is a locally cocartesian fibration. Prerequisitely:

\begin{lem}
$\ol{\Fun(\mb{C}, \mb{D})^\ot} \to \F$ is an inner fibration.
\end{lem}
\begin{proof}
Suppose we have $0 < i < n$ and a diagram
\[\xymatrix{  \Lambda^n_i \ar[r]^{k_0 \ \ \ \ \ \ } \ar[d] & \ol{\Fun(\mb{C}, \mb{D})^\ot} \ar[d] \\
\D^n \ar[r]_{k_1} & \F. }\]
Giving a lift of this diagram is equivalent to lifting the diagram
\[\xymatrix{\mb{C}^\ot_{k_0} \ar[r] \ar[d] & \mb{D}^\ot \ar[d] \\
\mb{C}^\ot_{k_1} \ar[r] & \F.
}\]
In the statement of \cite[Proposition 3.3.1.3]{HTT}, we can replace ``cartesian" with ``cocartesian" just by taking opposites. We deduce that the cofibration $\mb{C}^\ot_{k_0} \to \mb{C}^\ot_{k_1}$ is a categorical equivalence and therefore inner anodyne, permitting the lift.
\end{proof}
In particular, $\ol{\Fun(\mb{C}, \mb{D})^\ot}$ is a quasicategory. 
\begin{lem} \label{loccoc}
$\ol{\Fun(\mb{C}, \mb{D})^\ot}$ is a locally cocartesian fibration, and a morphism $(f : S \to T, F_f : \mb{C}^\ot_f \to \mb{D}^\ot_f) $ of $\ol{\Fun(\mb{C}, \mb{D})^\ot}$  is locally cocartesian iff the diagram
\[\xymatrix{
\mb{C}^\ot_S \ar[r]^{F_0} \ar[d] & \mb{D}^\ot_f  \ar[d]^{p} \\
\mb{C}^\ot_f \ar[ur]^{F_f} \ar[r] & \D^1}\]
exhibits $F_f$ as a $p$-left Kan extension of $F_0$, where $F_0$ is the composite of $F_S : \mb{C}^\ot_S \to \mb{D}^\ot_S$ with the natural inclusion $\mb{D}^\ot_S \inj \mb{D}^\ot_f$.
\end{lem}
\begin{proof}
Before trying too hard to prove this, it seems prudent to verify the following: 

\begin{lem}
The relative left Kan extensions arising in the statement of Lemma \ref{loccoc} actually exist.
\end{lem}
\begin{proof}
Applying \cite[Lemma 4.3.2.13]{HTT}, we must show that for each object $X \in \mb{C}^\ot_T$, the functor
\[(\mb{C}^\ot_S)_{/X} \to \mb{D}^\ot_f\]
admits a $p$-colimit. The proof of this is almost identical to that of \cite[Corollary 4.3.1.11]{HTT}; we simply replace the sentence ``Assumption (2) and Proposition 4.3.1.10 guarantee that $\ol{q}'$ is also a $p$-colimit diagram when regarded as a map from $K^\rhd$ to $X$" with ``The condition of Proposition 4.3.1.10 is vacuously satisfied for $\ol{q}'$, since there are no nonidentity edges with source $\{1\}$ in $\D^1$". 
\end{proof}
Let $\eps$ be the map $\La^n_0 \to \D^1$ which maps $0$ to $0$ and all other vertices to $1$. A map $\La^n_0 \to \ol{\Fun(\mb{C}, \mb{D})^\ot}$ lifting $\eps$ whose leftmost edge is $F_f$ gives a diagram
\[\xymatrix{
\del \D^{n-1} \ar[d] \ar[r] & \Fun_{\D^1}(\mb{C}^\ot_f, \mb{D}^\ot_f) \ar[d] \\
\D^{n -1} \ar[r] & \Fun(\mb{C}^\ot_S, \mb{D}^\ot_S)}\]
where the bottom horizontal map is the constant map at $F_0$. By \cite[Lemma 4.3.2.12]{HTT}, this diagram admits a lift. Since lifting left horn inclusions that factor through $\epsilon$ is sufficient to show that an edge is locally cocartesian, $\phi$ is locally cocartesian.

Thus we have shown that $\ol{\Fun(\mb{C}, \mb{D})^\ot}$ is a locally cocartesian fibration, and for each morphism $f$ of $\F$, we can choose a locally cocartesian edge $s$ over $f$ which corresponds to a relative left Kan extension as in Lemma \ref{loccoc}. Suppose $s'$ is another locally cocartesian edge over $f$ with the same source as $s$. Then $s$ and $s'$ are equivalent as edges of $\ol{\Fun(\mb{C}, \mb{D})^\ot}$, and so they must correspond to equivalent functors $\mb{C}^\ot_f \to \mb{D}^\ot_f$. Since one of these is relatively left Kan extended from $\mb{C}^\ot_S$, so must the other be. This proves the converse.
\end{proof}

What we've said so far makes sense for an arbitrary pair of cocartesian fibrations over an arbitrary base. Our construction gives rise to a locally cocartesian fibration in this generality, but there's no reason to expect it to be cocartesian. In fact, this won't be the case until we've cut $\ol{\Fun(\mb{C}, \mb{D})^\ot}$ down to an object which behaves sensibly with respect to the product decompositions occuring in $\mb{C}$, and then not without an additional condition on $\mb{D}^\ot$. 
\begin{rec}
For $\mb{C}^\ot$ a symmetric monoidal category, there is a canonical product decomposition 
\[\mb{C}^\ot_S \cong \mb{C}^S,\] 
given by the inert maps, where by $\mb{C}^S$, we mean the cartesian power indexed by the set $S^o$ of non-basepoint elements of $S$. In fact, more is true. Let $\F_{/S}^{act}$ be the full subcategory of $\F_{/S}$ spanned by the active morphisms. Then for any $S \in \F$, the obvious product decomposition
\[\F_{/S} \simeq \left(\prod_{s \in S^o} \F_{/\{s\}_+}^{act}\right) \X \F\]
given by taking preimages of each point of $S^o$ and the basepoint underlies a product decomposition
\[\mb{C}^\ot \X_\F \F_{/S} \simeq \left( \prod_{s \in S^o} \mb{C}^\ot \X_\F \F_{/\{s\}_+}^{act} \right) \X \mb{C}^\ot.\] 
In particular, for any morphism $f : S \to T$ in $\F$, there is a canonical fiber product decomposition
\[\mb{C}^\ot_f \cong \left( \prod_{\D^1,\, t \in T^o} \mb{C}^\ot_{\mu_{f^{-1}(t)_+}} \right) \X \mb{C}^\ot_{\beta_{f^{-1}(*)}},\]
where for a finite pointed set $V$, $\mu_V$ denotes the active map $V \to \angs{1}$ if $V$ is nonempty and the inclusion $\ast \inj \angs{1}$ if $V$ is empty, and $\beta_V$ denotes the unique map $V \to \ast$. This decomposition is compatible with the decompositions of $\mb{C}^\ot_S$ and $\mb{C}^\ot_T$.
\end{rec}

\begin{lem} \label{ghytna}
Let $\mb{D}^\ot$ be a symmetric monoidal $\iy$-category whose underlying category $\mb{D}$ admits all colimits. The following conditions on $\mb{D}^\ot$ are equivalent:
\begin{enumerate}[(i)]
\item The tensor product on $\mb{D}$ preserves colimits separately in each variable. That is, for each object $X \in \mb{D}$, the composite
\[\mb{D} \os{(X, -)} \to \mb{D} \X \mb{D} \os{\mu} \to \mb{D}\]
is a colimit-preserving functor.
\item Let $f : S \to T$ in $\F$ be a morphism, $(K_s)_{s \in S^o}$ an $S^o$-tuple of simplicial sets, and for each $s$, let $\phi_s : K_s \to \mb{D}$ be a functor. Let 
\[K = \prod_s K_s,\]
and using the product decomposition of $\mb{D}^\ot_S$, let 
\[\phi : K \to \mb{D}^\ot_S\]
be the product of the $\phi_s$. Suppose
\[\phi^\rhd : K^\rhd \to \mb{D}_S^\ot\]
is such that for each $s \in S^o$ and for each $y \in \prod_{s' \neq s} K_{s'}$, the composite
\[K_s^\rhd \os{(-, y)} \to K^\rhd \os{\phi^\rhd} \to \mb{D}_S^\ot \os{\pi_s} \to \mb{D}\]
is a colimit diagram, where $\pi_s$ is projection onto the $s$th factor. Then the cocartesian pushforward $f_*(\phi^\rhd)$ has the same property: for each $t \in T^o$ and for each $z \in \prod_{f(s) \neq t}$, the composite
\[\left( \prod_{f(s) = t} K_s \right)^\rhd \os{(-, z)} \to K^\rhd \os{f_*(\phi^\rhd)} \to \mb{D}_T^\ot \os{\pi_t} \to \mb{D}\]
is a colimit diagram.
\end{enumerate}
\end{lem}

\begin{proof}
One implication is obvious: letting $f$ be the active map $\angs{2} \to \angs{1}$ and $K_1 = \ast$ deduces (i) from (ii).

For the reverse impliciation, we may assume without loss of generality that $f$ is active and its target is $\angs{1}$. By factorizing $f$ as a composite of active maps for which the cardinality of each preimage is at most 2, we can reduce to the case where $f$ is the active morphism from $\angs{2}$ to $\angs{1}$. Now observe that $f_*$ preserves the left Kan extension along the projection $p: K \to K_2$, since each colimit arising in the Kan extension is under the aegis of (i). But this reduces to the case of (i), completing the proof.
\end{proof}

We'll assume that $\mb{D}$ satisfies these equivalent conditions for the remainder of the paper. With this in hand, we can make our main definition.

\begin{dfn}
Suppose $\mb{D}^\ot$ is such that the tensor product preserves colimits in each variable separately. The \emph{Day convolution symmetric monoidal $\infty$-category} $\Fun(\mb{C}, \mb{D})^\ot$ is the largest simplicial subset of $\ol{\Fun(\mb{C}, \mb{D})^\ot}$ whose vertices over $S \in \F$ are those corresponding to functors $F: \mb{C}^\ot_S \to \mb{D}^\ot_S$ which are in the essential image of the natural inclusion
\[\iota_S: \Fun(\mb{C}, \mb{D})^S \to \Fun(\mb{C}^S, \mb{D}^S) \cong \Fun(\mb{C}^\ot_S, \mb{D}^\ot_S)\]
and all of whose edges over $f : S \to T \in \F$ correspond to functors $F : \mb{C}^\ot_f \to \mb{D}^\ot_f$ which are in the essential image of the natural inclusion
\begin{align*} \iota_f : \prod_{t \in T^o} \Fun_{\D^1}(\mb{C}^\ot_{\mu_{f^{-1}(t)_+}}, \mb{D}^\ot_{\mu_{f^{-1}(t)_+}}) \X \Fun_{\D^1}(\mb{C}^\ot_{\beta_{f^{-1}(*)}}, \mb{D}^\ot_{\beta_{f^{-1}(*)}}) \\ \to \Fun_{\D^1}(\mb{C}^\ot_f, \mb{D}^\ot_f).\end{align*}
\end{dfn}

The fiber of this category over $S \in \F$ is evidently $\Fun(\mb{C}, \mb{D})^S$, so the symmetric monoidal $\infty$-category stakes are looking favorable. We need to verify a couple of things.

\begin{lem}
The natural projection $p : \Fun(\mb{C}, \mb{D}) \to \F$ is an inner fibration. Thus $\Fun(\mb{C}, \mb{D})^\ot$ is a subcategory of $\ol{\Fun(\mb{C}, \mb{D})^\ot}$.
\end{lem}
\begin{proof}
We only need to verify that the inner horn inclusion $\La^2_1 \to \D^2$ can be lifted against $p$. Let $f : S \to T$ and $g: T \to U$ be morphisms in $\F$, and let $\nu : \La^2_1 \to \F$ and $\rho: \D^2 \to \F$ be the corresponding maps. Let $\ol{\nu} : \La^2_1 \to \Fun(\mb{C}, \mb{D})^\ot$ be a lift of $\nu$; $\ol{\nu}$ is the data of a map
\[\kappa : \mb{C}^\ot_\nu \to \mb{D}^\ot_\nu\]
over $\La_1^2$. But the conditions satisfied by $\kappa$ imply that it decomposes, up to equivalence, as a product of maps
\[\kappa = \prod_{\La^2_1, \, u \in U} \kappa_u : \mb{C}^\ot_{\nu_u} \to \mb{D}^\ot_{\nu_u}\]
where $\nu_u : \La^2_1 \to \F$ is specified by the morphisms 
\[f_u : (gf)^{-1}(u) \to g^{-1}(u), \, g_u : g^{-1}(u) \to \{u\}_+\]
if $u \neq *$, and
\[f_u : (gf)^{-1}(u) \to g^{-1}(u), \, g_u : g^{-1}(u) \to *\]
if $u = *$.

Each $\kappa_u$ corresponds to a diagram
\[\begin{tikzcd}
\La^2_1 \ar{r}{\ol{\nu}_u} \ar{d} & \Fun(\mb{C}, \mb{D})^\ot \ar{d} \ar{r} & \ol{\Fun(\mb{C}, \mb{D})^\ot} \ar{dl} \\
\D^2 \ar{r}{\rho_u} & \F.
\end{tikzcd} \]
We can lift this diagram to get a functor 
\[\ol{\rho_u} : \D^2 \to \ol{\Fun(\mb{C}, \mb{D})^\ot}\]
which trivially factors through $\Fun(\mb{C}, \mb{D})^\ot$. This corresponds to a functor
\[\la_u : \mb{C}^\ot_{\rho_u} \to \mb{D}^\ot_{\rho_u}\]
over $\D^2$. Taking the product
\[\la = \prod_{\D^2, \, u \in U} \la_u\]
gives a solution to the original lifting problem.
\end{proof}

\begin{lem}\label{daycocart}
$p : \Fun(\mb{C}, \mb{D})^\ot \to \F$ is a cocartesian fibration.
\end{lem}

\begin{proof}
First we identify the locally cocartesian edges of $p$. Let $f : S \to T$ be a morphism in $\F$. If $T = \ast$, or if $|T^o| = 1$ and $f$ is active, then 
\[\Fun(\mb{C}, \mb{D})^\ot_f = (\ol{\Fun(\mb{C}, \mb{D})^\ot})_f\]
 and the locally cocartesian edges are $p$-left Kan extensions as before. Otherwise, a locally cocartesian edge is, up to equivalence, a product of these $p$-left Kan extensions along the product decompositions of $\mb{C}^\ot_f$ and $\mb{D}^\ot_f$.

It remains to show that the composition of locally cocartesian edges is locally cocartesian. Let $f: S \to T$, $g: T \to U$ be morphisms in $\F$; by the product decomposition of $\mb{C}^\ot_{\D^2}$, we may assume that $U = \angs{1}$, that $g$ is active, and that $f^{-1}(\ast) = \{\ast\}$. Let $(F_s)_{s \in S^o}$ be an $S^o$-tuple of functors $\mb{C} \to \mb{D}$, and let $X \in \mb{C}$ be an object. Then we must show that the canonical map
\[\colim{(Y_s) \in (\mb{C}^\ot_S)_{/X}} (gf)_*(F_s(Y_s)) \to \colim{(Z_t) \in (\mb{C}^\ot_T)_{/X}} g_*\left(\colim{(W_s) \in \left(\mb{C}^\ot_{f^{-1}(t)}\right)_{/Z_t}} f_*(F_s(W_s))\right)_t \]
is an equivalence. But the hypothesis that the tensor product on $\mb{D}$ preserves colimits in each variable separately precisely implies, by Lemma \ref{ghytna}, that
\begin{align*}
\colim{(Z_t) \in (\mb{C}^\ot_T)_{/X}} \left(\colim{(W_s) \in \left(\mb{C}^\ot_{f^{-1}(t)}\right)_{/Z_t}} (gf)_*(F_s(W_s))\right)_t \simeq \\
\colim{(Z_t) \in (\mb{C}^\ot_T)_{/X}} g_*\left(\colim{(W_s) \in \left(\mb{C}^\ot_{f^{-1}(t)}\right)_{/Z_t}} f_*(F_s(W_s))\right)_t \end{align*}
which is the same thing.
\end{proof} 

\begin{prop}
$\Fun(\mb{C}, \mb{D})^\ot \to \F$ is a symmetric monoidal $\infty$-category.
\end{prop}

\begin{proof}
We must verify the Segal condition: for each $n \in \bb{N}$, the pushforwards associated to the $n$ inert morphisms $\angs{n} \to \angs{1}$ exhibit $\mc{X}_{\angs{n}}$ as the product of $n$ copies of $\mc{X}_{\angs{1}}$. If, as in our case, one already has an identification of $\mc{X}_{\angs{n}}$ with $\mc{X}_{\angs{1}}^n$ up one's sleeve, then another way to express this condition is to say that the pushforward
\[i_j : \mc{X}_{\angs{1}}^n \to \mc{X}_{\angs{1}}\]
associated to the inert map $\chi_j : \angs{n} \to \angs{1}$ that picks out $j$ is equivalent to projection onto the $j$th factor. In this case our usual product decomposition takes the form
\[\mb{C}^\ot_{\chi_j} \simeq (\mb{C} \X \D^1) \X_{\D^1} \prod_{\D^1, \,i \neq j} (\mb{C}^\rhd)\]
and the conclusion follows immediately from the characterization of the locally cocartesian arrows in Lemma \ref{daycocart}. \end{proof}

\begin{prop} \label{laxmon}
A commutative monoid in $\Fun(\mb{C}, \mb{D})^\ot$ is exactly a lax symmetric monoidal functor from $\mb{C}^\ot$ to $\mb{D}^\ot$. \end{prop}
\begin{proof}
Recall that by definition, a lax symmetric monoidal functor from $\mb{C}^\ot$ to $\mb{D}^\ot$ is a morphism of $\iy$-operads from $\mb{C}^\ot$ to $\mb{D}^\ot$; that is, it is a morphism of categories over $\F$ that preserves cocartesian edges over inert morphisms. What we must prove is that that if $s : \F \to \Fun(\mb{C}, \mb{D})^\ot$ is a commutative monoid - that is, a section of the structure map which takes inert morphisms to cocartesian edges - the corresponding functor
\[\mb{C}^\ot \X_\F \F \simeq \mb{C}^\ot \to \mb{D}^\ot\]
preserves cocartesian edges over inert morphisms. So if $f$ is an inert map in $\F$, we must show that a functor $\mb{C}^\ot_f \to \mb{D}^\ot_f$ which corresponds to a cocartesian edge of $\Fun(\mb{C}, \mb{D}^\ot)$ preserves cocartesian edges over $f$.

By using the product decomposition of $\mb{C}^\ot_f$ for $f$ inert, we reduce to the following:
let $G : \mb{C} \to \mb{D}$ be a functor, and let $F_0$ be the composite
\[\mb{C} \os{G} \to \mb{D} \os{i_0} \to \mb{D} \X \D^1.\]
Then we must prove a functor $F : \mb{C} \X \D^1 \to \mb{D} \X \D^1$ is a left Kan extension of $F_0$ (relative to $\D^1$) iff it preserves cocartesian edges. Since $G \X \D^1$ is a left Kan extension of $F_0$, the former condition merely states that $F$ is equivalent to $G \X \D^1$, and $G \X \D^1$ clearly preserves cocartesian edges. To prove the converse, we find it most convenient to deploy some machinery. By the opposite of \cite[Proposition 3.1.2.3]{HTT} applied to the (opposite) marked anodyne map $\D^0 \to (\D^1)^\sharp$ and the cofibration $\emptyset \to \mb{C}^\flat$, the inclusion of marked simplicial sets
\[\mb{C}^\flat \to (\mb{C} \X \D^1)^\natural\]
is opposite marked anodyne and therefore a cocartesian equivalence in $\text{sSet}^+_{\D^1}$. This means that $F_0$ extends homotopy uniquely to a map of marked simplicial sets
\[(\mb{C} \X \D^1)^\natural \to (\mb{D} \X \D^1)^\natural\]
which proves the result.
\end{proof}

\begin{lem} \label{cpayan}
For any $\mb{C}^\ot$ and any $\mb{D}^\ot$ for which the tensor product commutes with colimits in each variable separately, $\Fun(\mb{C}, \mb{D})^\ot$ also has the property that the tensor product commutes with each variable separately.
\end{lem}
\begin{proof}
Suppose $\phi : K \to \Fun(\mb{C}, \mb{D})$ is a diagram and $F: \mb{C} \to \mb{D}$ is a functor. $F$ and $\phi$ define a functor $\psi : K \to \Fun(\mb{C}, \mb{D})^2 = \Fun(\mb{C}, \mb{D})^\ot_{\angs{2}}$, and we aim to show that the natural morphism
\[G:  \colim{K} (\mu_* \psi) \to \mu_*(\colim{K} \psi) \]
is an equivalence, where $\mu : \angs{2} \to \angs{1}$ is the active morphism. Evaluating each side on the object $X \in \mb{C}$ specializes $G$ to
\begin{align*} G_X : \colim{k \in K}  \colim{(Y, Z) \in \mb{C}^2 \X_\mb{C} \mb{C}_{/X}} (F(Y) \ot \phi(k)(Z)) \\ \to \colim{(Y, Z) \in \mb{C}^2 \X_\mb{C} \mb{C}_{/X}}(F(Y) \ot \colim{k \in K} \phi(k)(Z)).\end{align*}
By the hypothesis on $\mb{D}$, we can pull out the colimit over $K$ on the right and conclude that $G_X$ is an equivalence.
\end{proof}

In \cite[Corollary 6.3.1.12]{HA}, Lurie describes a Day convolution symmetric monoidal structure on the category of presheaves on a small symmetric monoidal category, which for consistency ought to agree with our construction in the relevant case.

\begin{prop}
Let $\mb{C}$ be a symmetric monoidal category and let $\mc{P}(\mb{C})^\ot$ denote the construction of \cite[6.3.1.12]{HA}. Endow $\Top$ with its product symmetric monoidal structure. Then there is a model for the symmetric monoidal category $(\mb{C}^\op)^\ot$ such that there is an equivalence of symmetric monoidal categories
\[\mc{P}(\mb{C})^\ot \cong \Fun(\mb{C}^\op, \Top)^\ot.\]
\end{prop}

In order to prove this, one only needs to show that $\Fun(\mb{C}^\op, \Top)^\ot$ satisfies the two criteria of \cite[Corollary 6.3.1.12]{HA}. Criterion (2) follows immediately from Lemma \ref{cpayan}, and Criterion (1) will be the subject of the next section.

\section{The symmetric monoidal Yoneda embedding}\label{Fonata}

A functor
\[\mb{C}^\ot \to \Fun(\mb{C}^\op, \Top)^\ot.\]
is the same thing as a functor
\[\mb{C}^\ot \X_{\F_*} (\mb{C}^\op)^\ot \to \Top^\X\]
satisfying certain conditions; this in turn is the same as a functor
\[\mb{C}^\ot \X_{\F_*} (\mb{C}^\op)^\ot \X_{\F_*} \Gamma^\X \to \Top\]
satisfying additional conditions, where $\Gamma^\X$ is the category of \cite[Notation 2.4.1.2]{HA} (see also \cite[Construction 2.4.1.4]{HA}).
Constructing this functor is going to require a small dose of extra technology. In Proposition \ref{(-)_*}, we'll describe a construction, due to Denis Nardin, of a strongly functorial pushforward for cocartesian fibrations. 

\begin{prop} \label{(-)_*}
Let $p: \mb{X} \to \mb{B}$ be any cocartesian fibration of $\iy$-categories. Let $\mc{O}_\mb{B}$ be the arrow category $\Fun(\D^1, \mb{B})$, equipped with its source and target maps
\[s, t: \mc{O}_\mb{B} \to \mb{B},\]
and form the pullback
\[\mb{X} \X_\mb{B} \mc{O}_\mb{B}\]
via the source map. Then there is a functor 
\[(-)_* : \mb{X} \X_\mb{B} \mc{O}_\mb{B} \to \mb{X}\]
which maps the object $(x, f)$ to $f_* x$ and makes the diagram
\[\xymatrix{
\mb{X} \X_\mb{B} \mc{O}_\mb{B} \ar[d]_t \ar[r]^{\hspace{15pt} (-)_*} & \mb{X} \ar[dl] \\
\mb{B}
}\]
commute.
\end{prop}

\begin{proof}
Let $\mc{O}_\mb{X}^c$ be the full subcategory of $\mc{O}_\mb{X}$ spanned by the cocartesian arrows and let $p : \mc{O}_\mb{X}^c \inj \mc{O}_\mb{B}$ be the projection. Then the essential point is that
\[(s, p) : \mc{O}_\mb{X}^c \to \mb{X} \X_\mb{B} \mc{O}_\mb{B}\]
is a trivial Kan fibration. Indeed, this follows from arguments made in \cite[\S 3.1.2]{HTT}, which we recall here for completeness: the square
\[\xymatrix{
\partial \D^n \ar[r] \ar[d] & \mc{O}_\mb{X}^c \ar[d] \\
\D^n \ar[r] & \mb{X} \X_\mb{B} \mc{O}_\mb{B}
}\]
determines the same lifting problem as the square
\[\xymatrix{
[(\del \D^n)^\flat \X (\D^1)^\sharp] \coprod_{(\del \D^n)^\flat \X \{0\}} [(\D^n)^\flat \X \{0\}] \ar[r] \ar[d] & \mb{X}^\natural \ar[d] \\
(\D^n)^\flat \X (\D^1)^\sharp \ar[r] &  \mb{B}^\sharp }\]
of marked simplicial sets. But the left vertical map is marked anodyne \cite[3.1.2.3]{HTT}, so a lift exists.

Now the diagram
\[\xymatrix{
\mc{O}^c_\mb{X} \ar[rrd]^s \ar[rdd]_{t \circ p} \ar[rd] \\
& \mb{X} \X_\mb{B} \mc{O}_\mb{B} \ar[d]_t & \mb{X} \ar[dl] \\
& \mb{B}}\]
commutes, so composing $s$ with any section of $(s, p)$ gives the desired map.
\end{proof}

We'll also need to import a result from a recent paper with Clark Barwick and Denis Nardin \cite{BGN}:
\begin{thm}
Let $p : \mb{X} \to \mb{B}$ be a cocartesian fibration of $\iy$-categories classified by a functor $F: \mb{B} \to \Cat_\iy$. Let $\op: \Cat_\iy \to \Cat_\iy$ be the functor that takes a category to its opposite. Then there is a cocartesian fibration $p' : \mb{X}' \to \mb{B}$ classified by $\op \circ F$ together with a functor $\Map_\mb{B}(-, -) : \mb{X}' \X_\mb{B} \mb{X} \to \Top$ such that for each $b \in \mb{B}$, the diagram
\[\begin{tikzcd}
\mb{X}' \X_{\mb{B}} \{b\} \X_{\mb{B}} \mb{X} \ar{r}[inner sep=5pt]{\Map_\mb{B}(-,-)} \ar{d}{\sim} & \Top \\
\mb{X}_b^\op \X \mb{X}_b \ar{ru}[swap]{\Map(-, -)}
\end{tikzcd}\]
homotopy commutes, where the vertical equivalence comes from the given identification of $\mb{X}'_b$ with $\mb{X}_b^\op$. Moreover, using these identifications, for each morphism $f: b \to b' \in \mb{B}$ and for each $(y, x) \in \mb{X}_b^\op \X \mb{X}_b$, $\xi$ sends (up to equivalence) the cocartesian edge from $(y, x)$ to $(F(f)(y), F(f)(x))$ to the natural map $\Map(y, x) \to \Map(F(f)(y), F(f)(x))$ induced by $F(f)$.
\end{thm}
\begin{proof}
Using the notation of \cite{BGN}, take $\mb{X}' = (\mb{X}^\vee)^\op$ and let $\Map_\mb{B}(-, -)$ be a functor classified by the left fibration $M : \td{\mc{O}}(\mb{X}/\mb{B}) \to \mb{X}' \X_\mb{B} \mb{X}$ of \cite[\S 5]{BGN}.
\end{proof}

Observe that Lurie's category $\Gamma^\X$ \cite[Notation 2.4.1.2]{HA} is the full subcategory of $\mc{O}_{\F_*}$ spanned by the inert morphisms. Thus by Theorem \ref{(-)_*} applied to $\mb{C}^\ot \X_{\F_*} (\mb{C}^\op)^\ot$, we get a functor
\[\phi_0' : \mb{C}^\ot \X_{\F_*} (\mb{C}^\op)^\ot \X_{\F_*} \Ga^\X \to \mb{C}^\ot \X_{\F_*} (\mb{C}^\op)^\ot.\]
Composing this with $\Map_{\F_*}(-, -) : \mb{C}^\ot \X_{\F_*} (\mb{C}^\op)^\ot \to \Top$ gives a functor
\[\phi_0 : \mb{C}^\ot \X_{\F_*} (\mb{C}^\op)^\ot \X_{\F_*} \Ga^\X \to \Top\]
which adjoints over to a functor
\[\phi : \mb{C}^\ot \X_{\F_*} (\mb{C}^\op)^\ot \to \td{\Top}^\X  \tag*{\cite[Construction 2.4.1.4]{HA}}\]
which factors through $\Top^\X$, by the properties of $(-)_*$ and $\Map_{\F_*}(-, -)$.

Adjointing again gives a functor
\[\psi : \mb{C}^\ot \to \ol{\Fun(\mb{C}^\op, \Top)^\ot}\]
We must check that $\psi$ factors through $\Fun(\mb{C}^\op, \Top)^\ot$. This is equivalent to the claim that for any morphism $f \in \F$, the pullback
\[\phi_f : \mb{C}^\ot_f \X_{\D^1} (\mb{C}^\op)^\ot_f \to \Top^\X_f\]
is compatible with the product decompositions of each side, which is readily verified.
We denote the ensuing functor
\[Y^\ot : \mb{C}^\ot \to \Fun(\mb{C}^\op, \Top)^\ot\]
and the final order of duty is to prove that $Y^\ot$ is symmetric monoidal. Let $\xi: (Z_s)_{s \in S^o} \to (W_t)_{t \in T^o}$ be a cocartesian morphism of $\mb{C}^\ot$ lying over the morphism $f : S \to T$ of $\F_*$, so that 
\[W_i \simeq \bigotimes_{j \in f^{-1}(i)} Z_j,\]
and let
\[Y_\xi : (\mb{C}^\op)^\ot_f \to \Top^\X_f\]
be the induced morphism. For each $(X_t)_{t \in T^o} \in (\mb{C}^\op)^\ot_T$, let

\[K = (\mb{C}^\op)^\ot_{/ (X_t)_{t \in T^o}} \X_{(\F_*)_{/T}} \{f\}\]
and let $\rho: K^\rhd \to (\mb{C}^\op)^\ot_f$ be the natural inclusion, taking the cone point to $(X_t)_{t \in T^o}$. Then the condition we must verify is that
\[Y_\xi \circ \rho : K^\rhd \to \Top^\X_f\]
is a colimit diagram relative to $\D^1$, which is to say that the natural map
\[\colim{((P_s)_{s \in S^o} \to (X_t)_{t \in T^o}) \in K} \, \, f_*((\Map(P_s, Z_s))_{s \in S^o}) \to (\Map(X_t, W_t))_{t \in T^o}\]
is an equivalence. To prove this, we may as well take $T = \angs{1}$ and $f$ to be active. Then we're really asking for the natural map
\[\colim{X \to P_1 \ot P_2 \ot \cdots \ot P_n}\prod_{i = 1}^n \Map(P_i, Z_i) \to \Map(X, \bigotimes_{i = 1}^n Z_i)\]
to be an equivalence for all $n\geq 0$,  $X \in \mb{C}^\op$ and $(Z_i)_{1 \leq i \leq n} \in \mb{C}^n$. 

Define a category $\mb{D}_n$ by the pullback square
\[\begin{tikzcd}
\mb{D}_n \ar{d} \ar{r} & (\mb{C}^n)_{/(Z_1, \cdots, Z_n)} \ar{d}{\mu_n \circ s} \\
\mb{C}_{X/} \ar{r}{t} & \mb{C}.
\end{tikzcd}\]
Then the functor on $\mb{C}^n \X_{\mb{C}} \mb{C}_{X/}$ taking $X \to P_1 \ot \cdots \ot P_n$ to $\prod_{i = 1}^n \Map(P_i, Z_i)$ is left Kan extended from the constant functor $\mb{D}_n \to \Top$ with image $*$, so we must show the ensuing map of spaces
\[N(\mb{D}_n) \to \Map(X, \bigotimes_{i = 1}^n Z_i)\]
is a weak equivalence. But replacing $(\mb{C}^n)_{/(Z_1, \cdots, Z_n)}$ with its final object \newline $\id_{(Z_1, \cdots, Z_n)}$ doesn't alter the homotopy type of the pullback, and the resulting pullback is $\Map(X, \bigotimes_{i = 1}^n Z_i)$ by definition. Unwinding the sequence of morphisms shows that this is the equivalence desired.

\RaggedRight
\bibliographystyle{alpha}
\bibliography{mybib}

\end{document}